\documentclass[11pt,epsf]{article}
\usepackage[dvips]{color}
\usepackage{geometry}
\newcommand{\bed}{\begin{displaymath}}
\newcommand{\eed}{\end{displaymath}}
\newcommand{\bea}{\bed\begin{array}{rl}}
\newcommand{\eea}{\end{array}\eed}
\newcommand{\barray}{\begin{array}{ll}}
\newcommand{\earray}{\end{array}}

\usepackage{color,latexsym,amsfonts,amssymb,amsmath,amsthm}
\usepackage{colordvi,amsfonts,amsmath,epsfig,subfigure,cite}
\usepackage{graphicx}
\usepackage{caption}
\newtheorem{theorem}{Theorem}[section]
\newtheorem{lemma}[theorem]{Lemma}

\newtheorem{proposition}[theorem]{Proposition}

\newtheorem{definition}{Definition}[section]
\newtheorem{rem}[theorem]{Remark}

\usepackage{fancyhdr}
\begin{document}

\title{S-asymptotically $\omega$-periodic solutions in distribution for a class of  stochastic fractional functional differential equations}
\author{Shufen Zhao$^{a,b}$
,~~Minghui Song$^{a}$\thanks{Corresponding author. songmh@hit.edu.cn. This work is supported by the NSF of P.R. China (No.11671113)}
\\$a$ Department of Mathematics, Harbin Institute of Technology, Harbin 150001, PR China\\
$b$ Department of  Mathematics, Zhaotong University, Zhaotong 657000,PR China}
\maketitle
\begin{abstract}
In this paper, we introduce the concepts of S-asymptotically $\omega$-periodic solutions in distribution for a class of stochastic fractional functional differential equations. The existence and uniqueness results for the S-asymptotically $\omega$-periodic solutions in distribution are obtained by means of the successive approximation and the Banach contraction mapping principle, respectively.


{\bf Mathematics Subject Classfication} 35B15; 34F05; 60H15.

\end{abstract}
\pagestyle{fancy}
\fancyhf{}
\fancyhead[CO]{ S-asymptotically $\omega$-periodic solutions in distribution}
\fancyhead[LE,RO]{\thepage}
\section{Introduction}
For the potential applications in theory and applications, the properties about almost automorphic, asymptotically almost automorphic, almost periodic, asymptotically almost
periodic and S-asymptotically $\omega$-periodic solutions of various determinate differential systems have been investigated by many researchers
(see e.g. \cite{yoshizawa2012stability,n2013almost,corduneanu1989almost,hino2001almost,henriquez2015almost,diagana2007existence,diagana2009existence,diagana2016existence,henriquez2016pseudo} and references therein). In the mean while, the corresponding concepts of stochastic differential equations are
also interesting topics in mathematical
analysis, for example, Fu and Liu \cite{fu2010square} introduced the concept of square-mean almost automorphy for stochastic processes and they studied square mean almost automorphic solutions to some linear and nonlinear stochastic differential equations. Cao et al. \cite{cao} introduced the concept of $p$-mean almost automorphy for stochastic processes.
Moreover, Fu in \cite{fu2012almost} introduced the concept of distributional almost automorphy for stochastic processes and obtained the existence and uniqueness of distributionally almost automorphic solutions to nonautonomous stochastic equations on
any real separable Hilbert space.
Liu and Sun \cite{sun} introduced the concepts of Poisson square-mean almost automorphy and almost automorphy in distribution and established the existence results of solutions that are almost automorphic in distribution for some semilinear stochastic differential equations with infinite dimensional L\'{e}vy noise and Li \cite{li2015weighted} considered the weighted pseudo almost automorphic solutions for nonautonomous
SPDEs driven by L\'{e}vy noise.

Henr\'{\i}quez et al. \cite{henriquez2008s} made an initial contribution to develop work in  \cite{zhongchao1987asymptotically,grimmer1969asymptotically,haiyin2006massera} and references therein to the theory of S-asymptotically $\omega$-periodic functions with values in a Banach space. Cuevas et al. \cite{cuevas2009s} considered the S-asymptotically $\omega$-periodic solution of the semilinear integro-differential equation of fractional order
\begin{eqnarray*}
\left\{ \begin{array}{ll}
x'(t)=\int_{0}^{t}\frac{(t-s)^{\alpha-2}}{\Gamma(\alpha-1)}A x(s)\mathrm{d}s+f(t,x(t)),\\
x(0)=c_{0}.\end{array} \right.
\end{eqnarray*}
Moreover, in \cite{cuevas2010existence}, Cuevas et al. considered the S-asymptotically $\omega$-periodic solution of the following form,
\begin{eqnarray*}
\left\{ \begin{array}{ll}
x'(t)=\int_{0}^{t}\frac{(t-s)^{\alpha-2}}{\Gamma(\alpha-1)}A x(s)\mathrm{d}s+f(t,x_{t}),\\
x(0)=\psi_{0}\in\mathcal{B},\end{array} \right.
\end{eqnarray*}
 where $\mathcal{B}$ is some abstract phase space.
Dimbour et al. \cite{dimbour2014s} considered the S-asymptotically $\omega$-periodic solutions of the differential equations with piecewise constant argument of the form
\begin{eqnarray*}
\left\{ \begin{array}{ll}
x'(t)=Ax(t)+A_{0}x([t])+g(t,x(t)),\\
x(0)=c_{0}.\end{array} \right.
\end{eqnarray*}

Inspired by the work mentioned above, in this paper, we investigate the existence of the S-asymptotically $\omega$-periodic solutions in distribution in an abstract space for a class of stochastic fractional functional differential equations driven by L\'{e}vy noise of the form
\begin{eqnarray}\label{1}
 \left\{ \begin{array}{ll}
\mathrm{d}D(t, x_{t})=\int_{0}^{t}\frac{(t-s)^{\alpha-2}}{\Gamma(\alpha-1)}AD(s, x_{s})\mathrm{d}s\mathrm{d}t+f(t,x_{t})\mathrm{d}t\\
\qquad+g(t,x_{t})\mathrm{d}w(t) +\int_{|u|_{U}<1}F(t,x(t^{-}),u)\tilde{N}(\mathrm{d}t,\mathrm{d}u)\\
\qquad+\int_{|u|_{U}\geq1}G(t,x(t^{-}),u)N(\mathrm{d}t,\mathrm{d}u),
x_{0}=\phi\in C_{\mathcal{F}_{0}}^{b}([-\tau,0],\mathbb{X}),\end{array} \right.
\end{eqnarray}
where $1<\alpha<2,$ $D(t,\varphi)=\varphi(0)+h(t,\varphi),$ $A:D(A)\subset\mathbb{X}\rightarrow\mathbb{X}$ is a linear densely defined operator of sectorial type on a Banach space $\mathbb{X}, $ $x_{t}$ be defined by $x_{t}(\theta)=x(t+\theta)$ for each $\theta\in [-\tau, 0],$ and $f,~g,~F,~G$ are functions subject to some additional conditions. The convolution integral in (\ref{1}) is known as the Riemann-Liouville fractional integral \cite{cuesta2003numerical,dos2010asymptotically}. We introduce the concept of Poisson square-mean S-asymptotically $\omega$-periodic solution for (\ref{1}) in order to correspond to the effect of the L\'{e}vy noise. Furthermore, we make an initial consideration of the S-asymptotically $\omega$-periodic solution in distribution in an abstract space $\mathcal{C}$ for (\ref{1}).

The paper is organized as follows. In Section 2, we review and introduce some concepts about square mean S-asymptotically $\omega$-periodic solutions in distribution for (\ref{1})
and some of their basic properties. We show the existence and uniqueness of the mild solution and the S-asymptotically $\omega$-periodic solution in distribution to (\ref{1}) in Section 3 and Section 4, respectively.
\section{Preliminaries}

Let $(\Omega,\mathcal{F},P)$ be a complete probability space equipped with some filtration $\{\mathcal{F}_{t}\}_{t\geq0}$ which satisfy the usual conditions, $(H,|\cdot|)$ and $(U,|\cdot|)$ are real separable Hilbert spaces. $\mathcal{L}(U,H)$ denote the space of all bounded linear operators from $U$ to $H$ which with the usual operator norm $\|\cdot\|_{\mathcal{L}(U,H)}$ is a Banach space. $L^{2}(P,H)$ is the space of all $H$-valued random variables $X$ such that $\mathbb{E}|X|^{2}=\int_{\Omega}|X|^{2}\mathrm{d}P<\infty.$ For $X\in L^{2}(P,H),$ let $\|X\|:=(\int_{\Omega}|X|^{2}\mathrm{d}P)^{1/2},$ it is well known that $(L^{2}(P,H),\|\cdot\|)$ is a Hilbert space. We denote by $\mathcal{M}^{2}([0,T],H)$ for the collection of stochastic processes $x(t):[0,T]\rightarrow L^{2}(P,H)$ such that $\mathbb{E}\int_{0}^{T}|x(s)|\mathrm{d}s<\infty.$ Let $\tau>0$ and $\mathcal{C}:=\mathcal{C}([-\tau,0];H)$ denote the family of all right-continuous functions with left hand limits $\varphi$ from $[-\tau, 0]$ to $H.$ The space $\mathcal{C}$ is assumed to be equipped with the norm $\|\varphi\|_{\mathcal{C}}=\sup_{-\tau\leq\theta\leq 0}|\varphi(\theta)|_{H}.$ $\mathcal{C}_{\mathcal{F}_{0}}^{b}([-\tau,0];H)$ denotes the family of all almost surely bounded $\mathcal{F}_{0}$-measurable, $\mathcal{C}$-valued random variables. $\{x_{t}\}_{t\in \mathbf{R}}$ is regarded as a $\mathcal{C}$-valued stochastic process.
In the following discussion, we always consider the L\'{e}vy processes that are $U$-valued.

\subsection{L\'{e}vy process}
Let $L$ is a L\'{e}vy process on $U,$ we write $\Delta L(t)=L(t)-L(t^{-})$ for all $t\geq 0.$  We define a counting Poisson random measure $N$ on $(U-\{0\})$ through $$N(t,O)=\sharp\{0\leq s\leq t:~\Delta L(s)(\omega')\in O\}=\Sigma_{0\leq s\leq t}\chi_{O}(\Delta L(s)(\omega'))$$ for any Borel set $O$ in $(U-\{0\}),$ $\chi_{O}$ is the indicator function. We write $\nu(\cdot)=E(N(1,\cdot))$ and call it the intensity measure associated with $L$. We say that a Borel set $O$ in $(U-\{0\}),$ is bounded below if $0\in \bar{O}$ where $\bar{O}$ is closure of $O.$ If $O$ is bounded below, then $N(t,O)<\infty$ almost surely for all $t\geq 0$ and $(N(t,O),~t\geq0)$ is a Poisson process with intensity $\nu(O).$ So $N$ is called Poisson random measure. For each $t\geq 0$ and $O$ bounded below, the associated compensated Poisson random measure $\tilde{N}$ is defined by $\tilde{N}(t,O)=N(t,O)-t\nu(O)$ (see \cite{applebaum2009levy,holden1996stochastic}).
\begin{proposition}(see \cite{applebaum2009levy})
(L\'{e}vy-It\^{o} decomposition).
If $L$ is a $U$-valued L\'{e}vy process, then there exist $a\in U,$ a $U$-valued Wiener process $w$ with covariance operator $Q,$ the so-called $Q$-wiener process, and an independent Poisson random measure $N$ on $\mathbf{R}^{+}\times (U-\{0\})$ such that, for each $t\geq0,$
\begin{equation}\label{ito}
  L(t)=at+w(t)+\int_{|u|_{U}<1}u\tilde{N}(t,\mathrm{d}u)+\int_{|u|_{U}\geq1}uN(t,\mathrm{d}u),
\end{equation}
where the Poisson random measure $N$ has the intensity measure $\nu$ which satisfies
$\int_{U}(|y|_{U}^{2}\wedge 1)\nu(\mathrm{d}y)<\infty$ and $\tilde{N}$ is the compensated Poisson random measure of $N.$
\end{proposition} The detail properties of L\'{e}vy process and $Q$-Wiener processes, we refer the readers to \cite{albeverio2002stochastic} and \cite{da2014stochastic}. Throughout the paper, we assume the covariance operator $Q$ of $w$ is of trace class, i.e. $TrQ<\infty$ and the L\'{e}vy process $L$ is defined on the filtered probability space $(\Omega,\mathcal{F},P,(\mathcal{F}_{t})_{t\in\mathbf{R}^{+}}).$ We also denote  $b:=\int_{|x|_{U}\geq 1}\nu(\mathrm{d}x)$ throughout the paper.
\subsection{Poisson square-mean S-asymptotically $\omega$-periodic solution of (\ref{1}) }
\begin{definition}(see \cite{sun})
A stochastic process $x:\mathbf{R}\rightarrow L^{2}(P,H)$ is said to be $L^{2}$-continuous if for any $s\in\mathbf{R},$ $\lim_{t\rightarrow s}\|x(t)-x(s)\|=0.$ It is $L^{2}$-bounded if $\|x\|_{\infty}=\sup_{t\in\mathbf{R}}\|x(t)\|<\infty.$
\end{definition}
Denote by $C_{b}(\mathbf{R};L^{2}(P,H))$ the Banach space of all $L^{2}$-bounded and $L^{2}$-continuous mapping from $\mathbf{R}$ to $L^{2}(P,H)$ endowed with the norm $\|\cdot\|_{\infty}.$
\begin{definition}
\begin{enumerate}
  \item [(1)] An $L^{2}$-continuous stochastic process $x:\mathbf{R}^{+} \rightarrow L^{2}(P,H)$ is said to be square-mean S-asymptotically $\omega$-periodic if there exists $\omega>0$ such that $\lim_{t\rightarrow\infty}\|x(t+\omega)-x(t)\|=0.$

      The collection of all S-asymptotically $\omega$-periodic stochastic processes $x:\mathbf{R}^{+}\rightarrow L^{2}(P,H)$ is denoted by $SAP_{\omega}(L^{2}(P,H)).$
  \item [(2)] A function $g:\mathbf{R}^{+} \times \mathcal{C}\rightarrow \mathcal{L}(U,L^{2}(P,H)),$ $(t,\varphi)\mapsto g(t,\varphi)$ is said to be square-mean S-asymptotically $\omega$-periodic in $t$ for each $\varphi\in \mathcal{C}$ if $g$ is continuous in the following sense $$\mathbb{E}\|(g(t,\phi)-g(t',\varphi))Q^{1/2}\|^{2}_{\mathcal{L}(U,L^{2}(P,H))}\rightarrow0~\text{as}~(t',\varphi)\rightarrow(t,\phi)$$
      and $$\lim_{t\rightarrow\infty}\mathbb{E}\|(g(t+\omega,\phi)-g(t,\phi))Q^{1/2}\|^{2}_{\mathcal{L}(U,L^{2}(P,H))}=0$$ for each $\phi\in \mathcal{C}.$
  \item [(3)] A function $F:\mathbf{R}^{+} \times \mathcal{C}\times U\rightarrow L^{2}(P,H),$ $(t,\phi,u)\mapsto F(t,\phi,u)$ with $ \int_{U}\|F(t,\phi,u)\|^{2}\nu(\mathrm{d}u)<\infty$ is said to be Poisson square-mean S-asymptotically $\omega$-periodic in $t$ for each $\phi\in\mathcal{C}$ if $F$ is continuous in the following sense $$\int_{U}\|F(t,\phi,u)-F(t',\varphi,u)\|^{2}\nu(\mathrm{d}u)\rightarrow 0~\text{as}~(t',\varphi)\rightarrow(t,\phi)$$ and that
       \begin{equation*}
      \lim_{t\rightarrow\infty}\int_{U}\|F(t+\omega,\phi,u)-F(t,\phi,u)\|^{2}\nu(\mathrm{d}u)=0
      \end{equation*}
      for each $\phi\in \mathcal{C}.$
\end{enumerate}
\end{definition}
\begin{rem}
Any square-mean S-asymptotically $\omega$-periodic process $x(t)$ is $L^{2}$-bounded and, by \cite{henriquez2008s}, $SAP_{\omega}(L^{2}(P,H))$ is a Banach space when it is equipped with the norm $$\|x\|_{\infty}:=\sup_{t\in\mathbf{R}^{+}}\|x(t)\|=\sup_{t\in\mathbf{R}^{+}}(\mathbb{E}|x(t)|^{2})^{\frac{1}{2}}.$$
\end{rem}
For the sequel, we introduce some definitions about square-mean S-asymptotically $\omega$-periodic functions with parameters.
\begin{definition}
\begin{enumerate}
  \item [(1)]A function $f:\mathbf{R}^{+}\times \mathcal{C}\rightarrow L^{2}(P,H)$ is said to be uniformly square-mean S-asymptotically $\omega$-periodic in $t$ on bounded sets if for every bounded set $K$ of $\mathcal{C}$, we have $\lim_{t\rightarrow\infty}\|f(t+\omega,\phi)-f(t,\phi)\|=0$ uniformly on $\phi\in \mathcal{C}.$
  \item [(2)] A function $g:\mathbf{R}^{+} \times \mathcal{C}\rightarrow \mathcal{L}(U,L^{2}(P,H)),$ is said to be uniformly square-mean S-asymptotically $\omega$-periodic on bounded sets if for every bounded set $K$ of $\mathcal{C}$, we have
      \begin{equation*}
        \lim_{t\rightarrow\infty}\mathbb{E}\|(g(t+\omega,\phi)-g(t,\phi))Q^{1/2}\|^{2}_{\mathcal{L}(U,L^{2}(P,H))}=0
      \end{equation*}
    uniformly on $\phi\in K.$
  \item [(3)] A function $F:\mathbf{R}^{+} \times \mathcal{C}\times U\rightarrow L^{2}(P,H)$ with $ \int_{U}\|F(t,\phi,u)\|^{2}\nu(\mathrm{d}u)<\infty$ is said to be uniformly Poisson square-mean S-asymptotically $\omega$-periodic in $t$ on bounded sets if for every bounded set $K$ of $\mathcal{C}$
      \begin{equation*}
        \lim_{t\rightarrow\infty}\int_{U}\|F(t+\omega,\phi,u)-F(t,\phi,u)\|^{2}\nu(\mathrm{d}u)=0
      \end{equation*}
     uniformly on $\phi\in K.$
\end{enumerate}
\end{definition}

\begin{lemma}\label{lem1}
Let $f:\mathbf{R}^{+}\times \mathcal{C}\rightarrow L^{2}(P,H),~(t,\phi)\mapsto f(t,\phi)$ be uniformly square-mean S-asymptotically $\omega$-periodic in $t$ on bounded sets of $ \mathcal{C}$ and assume that $f$ satisfies the Lipschitz condition in the sense $\|f(t,\phi)-f(t,\varphi)\|^{2}\leq L\|\phi-\varphi\|_{\mathcal{C}}^{2}$ for all $\phi,\varphi\in \mathcal{C}$ and $t\in \mathbf{R}, $ where $L$ is independent of $t.$ Then for any square-mean S-asymptotically $\omega$-periodic process $Y:\mathbf{R}\rightarrow L^{2}(P,H),$ the stochastic process $F:\mathbf{R}\rightarrow L^{2}(P,H)$ given by $F(t):=f(t,Y_{t})$ is square-mean S-asymptotically $\omega$-periodic.
\end{lemma}
\begin{proof}
Since $Y(t)\in SAP_{\omega}( L^{2}(P,H))$, the range of $Y(t)$ is a bounded set in $ L^{2}(P,H)$ which means that $\{Y_{t}\}_{t\in\mathbf{R}^+}$ is also a bounded set in $\mathcal{C}.$ Then $\lim_{t\rightarrow\infty}\|F(t+\omega,Y_{t+\omega})-F(t,Y_{t+\omega})\|=0.$
For any $\epsilon>0,$ $\exists ~T(\epsilon),$ such that $\|F(t+\omega,Y_{t+\omega})-F(t,Y_{t+\omega})\|<\epsilon/2$ and $\|Y_{t+\omega}-Y_{t}\|<\frac{\epsilon}{2L}.$
We get
\begin{eqnarray*}
  &&\|F(t+\omega)-F(t)\|\\&&\leq\|F(t+\omega,Y_{t+\omega})-F(t,Y_{t})\|\\
  &&=\|F(t+\omega,Y_{t+\omega})-F(t,Y_{t+\omega})\|+\|F(t+\omega,Y_{t+\omega})-F(t,Y_{t})\|\\
  &&\leq\epsilon,
\end{eqnarray*}
which completes the proof.
\end{proof}
\begin{lemma}\label{lem2}
Let $F:\mathbf{R}^{+} \times \mathcal{C}\times U\rightarrow  L^{2}(P,H)$ be uniformly Poisson square-mean S-asymptotically $\omega$-periodic in $t$ on bounded sets of $\mathcal{C}$ and $F$ satisfies the Lipschitz condition in the sense $$\int_{U}\|F(t,\phi,u)-F(t,\varphi,u)\|^{2}\nu(\mathrm{d}u)\leq L\|\phi-\varphi\|_{\mathcal{C}}^{2}$$ for all $\phi,\varphi\in \mathcal{C}$ and $t\in \mathbf{R}, $ where $L$ is independent of $t.$ Then for any square-mean S-asymptotically $\omega$-periodic process $Y(t):\mathbf{R}\rightarrow L^{2}(P,H),$ the stochastic process $\tilde{F}:\mathbf{R}\times U\rightarrow   L^{2}(P,H)$ given by $\tilde{F}(t,u):=F(t,Y_{t},u)$ is Poisson square-mean S-asymptotically $\omega$-periodic.
\end{lemma}
\begin{proof}
Since $F$ is uniformly Poisson square-mean S-asymptotically $\omega$-periodic in $t$ on bounded sets of $\mathcal{C}$ and
$Y(t)\in SAP_{\omega}(L^{2}(P,H),$ the range $\mathcal{R}(Y)$ of $Y(t)$ is a bounded set in $L^{2}(P,H),$
namely $\{Y_{t}\}_{t\in\mathbf{R}}$ is bounded in $\mathcal{C},$ we get  $\lim_{t\rightarrow\infty}\int_{U}\|F(t+\omega,\phi,u)-F(t,\phi,u)\|^{2}\nu(\mathrm{d}u)=0$ uniformly for $\phi\in \{Y_{t}\}_{t\in\mathbf{R}^{+}}.$ For any $\epsilon>0,$ we can find $T(\epsilon)>0$ such that when $t\geq T(\epsilon),$ we have $\int_{U}\|F(t+\omega,\phi,u)-F(t,\phi,u)\|^{2}\nu(\mathrm{d}u)\leq \epsilon/4,~\text{for~any}~\phi\in \{Y_{t}\}_{t\in\mathbf{R}^{+}}$ and $\|Y_{t+\omega}-Y_{t}\|^{2}<\frac{\epsilon}{2L}.$
 Note that
\begin{eqnarray*}
&&\tilde{F}(t+\omega,u)-\tilde{F}(t,u)\\
&&=F(t+\omega,Y_{t+\omega},u)-F(t+\omega,Y_{t},u)+F(t+\omega,Y_{t},u)-F(t,Y_{t},u),
\end{eqnarray*}
so for the above $\epsilon,$ when $t\geq T(\epsilon),$ we have
\begin{eqnarray*}
&&\int_{U}\|\tilde{F}(t+\omega,u)-\tilde{F}(t,u)\|^{2}\nu(\mathrm{d}u)\\
&&\leq 2\int_{U}\|F(t+\omega,Y_{t+\omega},u)-F(t,Y_{t+\omega},u)\|^{2}\nu(\mathrm{d}u)\\
&&~~+2\int_{U}\|F(t,Y_{t+\omega},u)-F(t,Y_{t},u)\|^{2}\nu(\mathrm{d}u)\\
&&\leq \epsilon/2+2L\|Y_{t+\omega}-Y_{t}\|^{2}\\
&&\leq \epsilon.
\end{eqnarray*}
We deduce that
\begin{equation*}
  \lim_{t\rightarrow\infty }\int_{U}\|\tilde{F}(t+\omega,u)-\tilde{F}(t,u)\|^{2}\nu(\mathrm{d}u)=0,
\end{equation*}
which means that $\tilde{F}(t,u)$ is Poisson square-mean S-asymptotically $\omega$-periodic.
\end{proof}
Let $\mathcal{P}(\mathcal{C})$ be the space of Borel probability measures on $\mathcal{C}$, for $P_{1},~P_{2}\in \mathcal{P}(\mathcal{C})$, denote metric $d_{L}$ as follows
\begin{equation*}
  d_{\mathfrak{L}}(P_{1},P_{2})=\sup_{f\in L}|\int_{\mathcal{C}}f(\sigma)P_{1}(\mathrm{d}\sigma)-\int_{\mathcal{C}}f(\sigma)P_{2}(\mathrm{d}\sigma)|,
\end{equation*} where
\begin{equation*}
  \mathfrak{L}=\{f:\mathcal{C}\rightarrow \mathbf{R}:~|f(\phi)-f(\varphi)|\leq \|\phi-\varphi\|_{C}~\text{and }~|f(\cdot)|\leq 1\}.
\end{equation*}
\begin{definition}
A stochastic process $x_{t}:\mathbf{R}\rightarrow \mathcal{C}$ is said to be S-asymptotically $\omega$-periodic in distribution if
the law $\mu(t)$  of $x_{t}$ is a $\mathcal{P}(\mathcal{C})$-valued S-asymptotically $\omega$-periodic mapping, i.e.
there is a positive number $\omega$ such that
\begin{equation*}
  \lim_{t\rightarrow\infty}d_{\mathfrak{L}}(\mu(t+\omega),\mu(t))=0.
\end{equation*}
\end{definition}
\begin{lemma}\label{lem3}
Any square-mean S-asymptotically $\omega$-periodic solution of (\ref{1}) is necessarily S-asymptotically $\omega$-periodic in distribution.
\end{lemma}
\begin{proof}
Let $x(t)\in SPA_{\omega}(L^{2}(P,H))$ be a solution of (\ref{1}), then there exists $\omega>0$ such that
\begin{equation}\label{re}
  \lim_{t\rightarrow\infty}\|x(t+\omega)-x(t)\|=0.
\end{equation}
We need to show that the law $\mu(t)$  of $x_{t}$ satisfies
 \begin{equation*}
  \lim_{t\rightarrow\infty}d_{\mathfrak{L}}(\mu(t+\omega),\mu(t))=0,
\end{equation*}
which is equivalent to show for any $\epsilon>0,$ there is a $T>0$ such that
\begin{equation*}
  \sup_{f\in \mathfrak{L}}|\int_{\mathcal{C}}f(\sigma)\mu(t+\omega)(\mathrm{d}\sigma)-\int_{\mathcal{C}}f(\sigma)\mu(t)(\mathrm{d}\sigma)|\leq\epsilon,~\forall t\geq T.
\end{equation*}
Since for any $f\in \mathcal{\mathfrak{L}},$
\begin{eqnarray*}
|\mathbb{E}(f(x_{t}^{x_{\omega}}))-\mathbb{E}(f(x_{t}^{\phi}))|\leq\mathbb{E}(2\wedge\|x_{t}^{x_{\omega}}-f(x_{t}^{\phi})\|_{\mathcal{C}}).
\end{eqnarray*}
From (\ref{re}), there is a $T_{\epsilon}>0$ satisfying $\mathbb{E}\|x_{t}^{x_{\omega}}-f(x_{t}^{\phi})\|_{\mathcal{C}}^{2}\leq \epsilon^{2}.$
For the arbitrary of $f\in\mathfrak{L},$ we get
\begin{equation*}
  \sup_{f\in \mathfrak{L}}|\int_{\mathcal{C}}f(\sigma)\mu(t+\omega)(\mathrm{d}\sigma)-\int_{\mathcal{C}}f(\sigma)\mu(t)(\mathrm{d}\sigma)|\leq\epsilon,~\forall t\geq T.
\end{equation*}
The proof is completed.
\end{proof}
\subsection{Sectorial operators}
We recall some definition about sectorial operators which have been studied well in the past decades, for details, see \cite{haase2006functional,lunardi2012analytic}.

\begin{definition}
Let $\mathbb{X}$ be an Banach space, $A:D(A)\subseteq \mathbb{X}\rightarrow \mathbb{X}$ is a close linear operator. $A$ is said to be a sectorial operator of type $\mu$ and angle $\theta$ if there exist $0<\theta<\pi/2,~M>0$ and $\mu\in\mathbf{R}$ such that the resolvent $\rho(A)$ of $A$ exists outside the sector $\mu+S_{\theta}=\{\mu+\lambda:\lambda\in\mathbf{C} ,|arg(-\lambda)|<\theta\}$ and $\|(\lambda-A)^{-1}\|\leq \frac{M}{|\lambda-\mu|}$ when $\lambda$ does not belong to $\mu+S_{\theta}.$
\end{definition}
\begin{definition}(see \cite{dos2010asymptotically})
Let $A$ be a closed and linear operator with domain $D(A)$ defined on a Banach space $\mathbb{X}.$ We call $A$ the generator of a solution operator if there exist $\mu\in \mathbf{R}$ and a strongly continuous function $S_{\alpha}:~\mathbf{R}^{+}\rightarrow\mathcal{L}(\mathbb{X},\mathbb{X})$ such that $\{\lambda^{\alpha}:~Re(\lambda)>\mu\}\subset\rho(A)$ and $\lambda^{\alpha-1}(\lambda^{\alpha}-A)^{-1}x=\int_{0}^{\infty}e^{-\lambda t}S_{\alpha}(t)\mathrm{d}t,$ $Re(\lambda)>\mu,$ $x\in \mathbb{X}.$ In this case, $S_{\alpha}(\cdot)$ is called the solution operator generated by $A.$
\end{definition}
If $A$ is sectorial of type $\mu$ with $1<\theta<\pi(1-\frac{\alpha}{2}),$ then $A$ is the generator of a solution operator given by $S_{\alpha}(t)=\frac{1}{2\pi i}\int_{\gamma }e^{\lambda t}\lambda ^{\alpha-1}(\lambda^{\alpha}-A)^{-1}\mathrm{d}\lambda,$ where $\gamma$ is a suitable path lying outside the sector $\mu+S_{\theta}$ \cite{bajlekova2001fractional}. Cuesta \cite{cuesta2007asymptotic} showed that if $A$ is a sectorial operator of type $\mu<0,$ for some $M>0$ and $0<\theta<\pi(1-\frac{\pi}{2}),$ there is $C>0$ such that
\begin{equation}\label{sem}
  \|S_{\alpha}(t)\|\leq\frac{CM}{1+|\mu|t^{\alpha}},~t\geq 0.
\end{equation}

\section{Existence of mild solution}
\begin{definition}\label{def1}
An $\mathcal{F}_{t}$-progressively measurable stochastic process $\{x(t)\}_{t\in\mathbf{R}}$ is called a mild solution of (\ref{1}) if it satisfies the corresponding stochastic integral equation
\begin{eqnarray}\label{milds}
\nonumber x(t)&=&S_{\alpha}(t)(\phi(0)+h(0,\phi))-h(t,x_{t})+\int_{0}^{t}S_{\alpha}(t-s)f(s,x_{s})\mathrm{d}s\\
&&+\int_{0}^{t}S_{\alpha}(t-s)g(s,x_{s})\mathrm{d}w(s)\\
\nonumber&&+\int_{0}^{t}\int_{|u|<1}S_{\alpha}(t-s)F(s,x(s^{-}),u)\tilde{N}(\mathrm{d}s,\mathrm{d}u)\\
\nonumber&&+\int_{0}^{t}\int_{|u|\geq1}S_{\alpha}(t-s)G(s,x(s^{-}),u)N(\mathrm{d}s,\mathrm{d}u),\\
\nonumber x_{0}&=&\phi(\cdot)\in C_{\mathcal{F}_{0}}^{b}([-\tau,0];H)
\end{eqnarray}
for all $t\geq 0.$
\end{definition}
In order to establish our main result, we impose the following conditions.
\begin{enumerate}
\item [(H1)] $A$ is a sectorial operator of type $\mu<0$ and angle $\theta$ with $0\leq \theta\leq \pi(1-\alpha/2).$
 \item [(H2)] $h(t,0)=0$ and for all $\varphi,\psi\in \mathcal{C}$ there exists a constant $k_{0}\in(0,1)$ such that $|h(t,\varphi)-h(t,\psi)|\leq k_{0}\|\varphi-\psi\|_{\mathcal{C}}.$
  \item [(H3)] $f:\mathbf{R}^{+}\times  \mathcal{C}\rightarrow L^{2}(P,H),$ $g:~\mathbf{R}^{+}\times  \mathcal{C}\rightarrow \mathcal{L}(U,L^{2}(P,H)),$ and $f(t,0)=0,$ $g(t,0)=0,$ $F:\mathbf{R}^{+}\times  \mathcal{C}\times U\rightarrow L^{2}(P,H),$ $G:\mathbf{R}^{+}\times  \mathcal{C}\times U\rightarrow L^{2}(P,H),$ $F(t,0,u)=0,$ $G(t,0,u)=0.$ For all $t\in \mathbf{R}^{+},$
  \begin{eqnarray*}
\|f(t,\phi)-f(t,\varphi)\|^{2} &\leq &L\|\phi-\varphi\|_{\mathcal{C}}^{2},\\
\mathbb{E}\|(g(t,\phi)-g(t,\varphi))Q^{1/2}\|^{2}_{\mathcal{L}(U,L^{2}(P,H))}&\leq& L\|\phi-\varphi\|_{ \mathcal{C}}^{2},\\
   \int_{|u|_{U}<1}\|F(t,\phi,u)-F(t,\varphi,u)\|^{2}\nu(\mathrm{d}u)&\leq& L\|\phi-\varphi\|_{ \mathcal{C}}^{2},\\
     \int_{|u|_{U}\geq1}\|G(t,\phi,u)-G(t,\varphi,u)\|^{2}\nu(\mathrm{d}u)&\leq& L\|\phi-\varphi\|_{ \mathcal{C}}^{2},
  \end{eqnarray*}
  for some constant $L>0$ independent of $t.$

   \end{enumerate}
\begin{theorem}
If (H1)-(H2) hold, then the Cauchy problem (\ref{1}) has a unique mild solution.
\end{theorem}
\begin{proof} 
Define $x_{0}^{0}=\phi$ and $x^{0}(t)=S_{\alpha}(t)(\phi(0)+f(0,\phi))$ for $t\geq0.$

Set $x^{n}_{0}=\phi,$ for $n=1,2,...,$ $\forall~ T\in (0,\infty),$ we define the sequence of successive approximations to (\ref{1}) as follows:
\begin{eqnarray}\label{p1}
\nonumber&& x^{n}(t)+h(t,x_{t}^{n})\\
&&=S_{\alpha}(t)(\phi(0)+h(0,\phi))+\int_{0}^{t}S_{\alpha}(t-s)f(s,x_{s}^{n-1}))\mathrm{d}s\\
\nonumber&&~+\int_{0}^{t}S_{\alpha}(t-s)g(s,x_{s}^{n-1})\mathrm{d}w(s)\\
\nonumber&&~+\int_{0}^{t}\int_{|u|<1}S_{\alpha}(t-s)F(s,x^{n-1}_{s^{-}},u)\tilde{N}(\mathrm{d}s,\mathrm{d}u)\\
\nonumber&&~+\int_{0}^{t}\int_{|u|\geq1}S_{\alpha}(t-s)G(s,x^{n-1}_{s^{-}},u)N(\mathrm{d}s,\mathrm{d}u),
\end{eqnarray}
for $t\in[0,T].$ Obviously, $x^{0}(\cdot)\in C_{b}(\mathbf{R}, L^{2}(P,H)),$ and $\|x^{0}(\cdot)\|^{2}_{\infty}\leq c'$ where $c'=2(CM)^{2}(1+k_{0}^{2})\|\phi\|_{\mathcal{C}}^{2}$ is a positive constant.
\begin{eqnarray*}
&&\mathbb{E}\sup_{0\leq s\leq t}|x^{n}(t)+h(t,x_{t}^{n})|^{2}\\
&&=5\mathbb{E}\sup_{0\leq s\leq t}|S_{\alpha}(t)(\phi(0)+h(0,\phi))|^{2}+5\mathbb{E}\sup_{0\leq s\leq t}|\int_{0}^{s}S_{\alpha}(s-r)f(r,x^{n-1}_{r})\mathrm{d}r|^{2}\\
&&~+5\mathbb{E}\sup_{0\leq s\leq t}|\int_{0}^{s}S_{\alpha}(s-r)g(r,x^{n-1}_{r})\mathrm{d}w(r)|^{2}\\
&&~+5\mathbb{E}\sup_{0\leq s\leq t}|\int_{0}^{s}\int_{|u|<1}S_{\alpha}(s-r)F(r,x^{n-1}_{r^{-}},u)\tilde{N}(\mathrm{d}r,\mathrm{d}u)|^{2}
\\&&~+5\mathbb{E}\sup_{0\leq s\leq t}|\int_{0}^{s}\int_{|u|\geq1}S_{\alpha}(s-r)G(r,x^{n-1}_{r^{-}},u)N(\mathrm{d}r,\mathrm{d}u)|^{2}\\
&&=\sum_{i=1}^{5}I_{i}
\end{eqnarray*}
Obviously, $I_{1}\leq 5c',$ and
\begin{eqnarray*}
I_{2}&=&5\mathbb{E}\sup_{0\leq s\leq t}|\int_{0}^{s}S_{\alpha}(s-r)f(r,x^{n-1}_{r})\mathrm{d}r|^{2}\\
&\leq&5\mathbb{E}\sup_{0\leq s\leq t}\int_{0}^{s}S_{\alpha}(s-r)\mathrm{d}r\int_{0}^{s}S_{\alpha}(s-r)|f(r,x_{r}^{n-1})|^{2}\mathrm{d}r\\
&\leq&5CM\frac{|\mu|^{-1/\alpha}\pi}{\alpha\sin(\pi/\alpha)}L\mathbb{ E}\int_{0}^{t}S_{\alpha}(t-r)\|x_{r}^{n-1}\|_{\mathcal{C}}^{2}\mathrm{d}r\\
&\leq&5(CM)^{2}\frac{|\mu|^{-1/\alpha}\pi}{\alpha\sin(\pi/\alpha)}L \mathbb{E}\int_{0}^{t}\|x_{r}^{n-1}\|_{\mathcal{C}}^{2}\mathrm{d}r,
\end{eqnarray*}
It follows from It\^{o}'s isometry that
\begin{eqnarray*}
I_{3}&=&5\mathbb{E}\sup_{0\leq s\leq t}|\int_{0}^{s}S_{\alpha}(s-r)g(r,x_{r}^{n-1})\mathrm{d}w(r)|\\
&\leq&5(CM)^{2}\mathbb{E}\int_{0}^{t}\|g(r,x^{n-1}_{r})Q^{1/2}\|^{2}_{\mathcal{L}(U,L^{2}(P,H))}\mathrm{d}r\\
&\leq&5(CM)^{2}L\mathbb{E}\int_{0}^{t}\|x_{r}^{n-1}\|_{\mathcal{C}}^{2}\mathrm{d}r
\end{eqnarray*}
By using the properties of integrals for poisson random measures, we get
\begin{eqnarray*}
I_{4}&=&\mathbb{E}\sup_{0\leq s\leq t}|\int_{0}^{s}\int_{|u|<1}S_{\alpha}(t-r)F(r,x^{n-1}_{r^{-}},u)\tilde{N}(\mathrm{d}r,\mathrm{d}u)|^{2}\\
&\leq&5(CM)^{2}L\mathbb{E}\int_{0}^{t}\|x_{r}^{n-1}\|_{\mathcal{C}}^{2}\mathrm{d}r
\end{eqnarray*}
\begin{eqnarray*}
I_{5}&=&5\mathbb{E}\sup_{0\leq s\leq t}|\int_{0}^{s}\int_{|u|\geq1}S_{\alpha}(s-r)G(r,x^{n-1}_{r^{-}},u)N(\mathrm{d}r,\mathrm{d}u)|^{2}\\
&\leq&10(CM)^{2}[\int_{0}^{t}\mathbb{E}\int_{|u\geq1|}(\frac{1}{1+|\mu|(t-s)^{\alpha}})^{2}|G(r,x^{n-1}_{r^{-}},u)|^{2}\nu(\mathrm{d}u)\mathrm{d}s\\
&&+\int_{0}^{t}\frac{1}{1+|\mu|(t-r)^{\alpha}}\int_{|u|\geq 1}\nu(\mathrm{d}u)\mathrm{d}s\mathbb{E}\int_{0}^{t}(\int_{|u|\geq1}|G(r,x^{n-1}_{r^{-}},u)|^{2}\nu(\mathrm{d}u))\mathrm{d}s]\\
&\leq& 10(CM)^{2}L(1+b\frac{|\mu|^{-1/\alpha}\pi}{\alpha\sin(\pi/\alpha)})\mathbb{E}\int_{0}^{t}\|x_{r}^{n-1}\|_{\mathcal{C}}^{2}\mathrm{d}r.
\end{eqnarray*}
Note that
\begin{eqnarray*}
|x^{n}(t)|^{2}&=&|x^{n}(t)+h(t,x^{n}_{t})-h(t,x^{n}_{t})|^{2}\\
&\leq&\frac{1}{1-k_{0}}(|x^{n}(t)+h(t,x^{n}_{t})|^{2}+k_{0}(1-k_{0})\|x_{t}^{n}\|^{2}_{\mathcal{C}}),
\end{eqnarray*}
by taking the expectation on both sides of the above inequality, we get
\begin{eqnarray*}
\mathbb{E}\sup_{0\leq s\leq t}|x^{n}(s)|^{2}\leq \frac{1}{1-k_{0}}\mathbb{E}\sup_{0\leq s\leq t}|x^{n}(s)+h(s,x^{n}_{s})|^{2}+k_{0}\mathbb{E}\sup_{0\leq s\leq t}\|x^{n}_{s}\|_{\mathcal{C}}^{2}.
\end{eqnarray*}
Combining the estimations for $I_{1}-I_{5},$ we get
\begin{eqnarray*}
&&\mathbb{E}\sup_{0\leq s\leq t}|x^{n}(s)|^{2}\\
&&\leq \frac{1}{1-k_{0}}5(CM)^{2}[(1+k_{0}^{2})\|\phi\|_{\mathcal{C}}^{2}]
+\frac{1}{1-k_{0}}5(CM)^{2}L\\&&~~\times\big(\frac{|\mu|^{-1/\alpha}\pi}{\alpha\sin(\pi/\alpha)}+2+2(1+b\frac{|\mu|^{-1/\alpha}\pi}{\alpha\sin(\pi/\alpha)}) \big)\mathbb{E}\int_{0}^{t}\|x_{r}^{n-1}\|_{\mathcal{C}}^{2}\mathrm{d}r\\
&&= c_{1}+c_{2} \mathbb{E}\sup_{0\leq s\leq t}\int_{0}^{s}\|x^{n-1}_{r}\|_{\mathcal{C}}^{2}\mathrm{d}r.
\end{eqnarray*}
Then for any arbitrary positive integer $\tilde{k},$ we have
\begin{eqnarray*}
&& \max_{1\leq n\leq \tilde{k}}\mathbb{E} \sup_{0\leq s\leq t}|x^{n}(s)|^{2}\\
&& \leq  c_{1}+c_{2}\int_{0}^{t}\|\phi\|^{2}_{\mathcal{C}}\mathrm{d}r+c_{2}\int_{0}^{t}\sup_{0\leq r\leq t} \max_{1\leq n\leq \tilde{k}}\mathbb{E}\|x^{n-1}(r)\|_{\mathcal{C}}^{2}\mathrm{d}r.
  \end{eqnarray*}
By the Gronwall inequality, we get
  $$\max_{1\leq n\leq \tilde{k}}\mathbb{E} \sup_{0\leq s\leq t} |x^{n}(s)|^{2}\leq (c_{1}+c_{2}\|\phi\|^{2}T)e^{c_{3}t}.$$
  Due to the arbitrary of $\tilde{k},$ we have
  \begin{equation}\label{squa}
    \mathbb{E}\sup_{0\leq s\leq T} |x^{n}(s)|^{2}\leq (c_{1}+c_{2}\|\phi\|_{\mathcal{C}}^{2}T)e^{c_{3}T}.
  \end{equation}
  So $x^{n}(t)\in \mathcal{M}^{2}([0,T],H).$
Obviously, we have
\begin{eqnarray*}
  &&\mathbb{E}(\sup_{0\leq s\leq t}|x^{n+1}(s)-x^{n}(s)|^{2})\\
  &&\leq\frac{1}{1-k_{0}} \mathbb{E}(\sup_{0\leq s\leq t}(|x^{n+1}(s)-x^{n}(s)+h(s,x^{n+1}_{s})-h(s,x^{n}_{s})|^{2}\\&&~+k_{0} \mathbb{E}(\sup_{0\leq s\leq t}|x^{n+1}(s)-x^{n}(s)|^{2}),
\end{eqnarray*}
namely,
\begin{eqnarray*}
  &&\mathbb{E}(\sup_{0\leq s\leq t}|x^{n+1}(s)-x^{n}(s)|^{2})\\&&\leq
  \frac{1}{(1-k_{0})^{2}} \mathbb{E}\sup_{0\leq s\leq t}|x^{n+1}(s)-x^{n}(s)+h(s,x^{n+1}_{s})-h(s,x^{n}_{s})|^{2},
\end{eqnarray*}
and
  \begin{eqnarray*}\label{fir1}
    &&\mathbb{E}\sup_{0\leq s\leq t}|[x^{n+1}(s)-x^{n}(s)]+[h(t,x^{n+1}_{s})-h(t,x^{n}_{s})]|^{2}\\
    &&\leq4\mathbb{E}\sup_{0\leq s\leq t}|\int_{0}^{s}S_{\alpha}(s-r)[f(r,x_{r}^{n}))-f(r,x_{r}^{n-1}))]\mathrm{d}r|^{2}\\
    &&~+4\mathbb{E}\sup_{0\leq s\leq t}|\int_{0}^{s}S_{\alpha}(s-r)[g(r,x_{r}^{n})-g(r,x_{r}^{n-1})]\mathrm{d}w(r)|^{2}\\
    &&~+4\mathbb{E}\sup_{0\leq s\leq t}|\int_{0}^{t}\int_{|u|<1}S_{\alpha}(s-r)[F(r,x^{n}_{r^{-}},u)-F(r,x^{n-1}_{r^{-}},u)]\tilde{N}(\mathrm{d}r,\mathrm{d}u)|^{2}\\
    &&~+4\mathbb{E}\sup_{0\leq s\leq t}|\int_{0}^{s}\int_{|u|\geq1}S_{\alpha}(s-r)[G(r,x^{n}_{r^{-}},u)-G(r,x^{n-1}_{r^{-}},u)]N(\mathrm{d}r,\mathrm{d}u)|^{2}.
    \end{eqnarray*}
By the fact
    \begin{eqnarray*}&&\mathbb{E}\sup_{0\leq s\leq t}|\int_{0}^{s}S_{\alpha}(s-r)[f(r,x_{r}^{n}))-f(r,x_{r}^{n-1}))]\mathrm{d}r|^{2}\\
   &&\leq L\mathbb{E}\sup_{0\leq s\leq t}\int_{0}^{s}S_{\alpha}(s-r)\mathrm{d}r\int_{0}^{s}S_{\alpha}(s-r)\|x_{r}^{n}-x_{r}^{n-1}\|_{\mathcal{C}}^{2}\mathrm{d}r\\
    &&\leq L(CM)^{2}\frac{|\mu|^{-1/\alpha}\pi}{\alpha\sin(\pi/\alpha)}\mathbb{E}\int_{0}^{t}\|x_{r}^{n}-x_{r}^{n-1}\|_{\mathcal{C}}^{2}\mathrm{d}r, \end{eqnarray*}
 \begin{eqnarray*}
 && \mathbb{E}\sup_{0\leq s\leq t}|\int_{0}^{s}S_{\alpha}(s-r)[g(r,x_{r}^{n})-g(r,x_{r}^{n-1})]\mathrm{d}w(r)|^{2}\\
  &&\leq\mathbb{E}\int_{0}^{t}S_{\alpha}^{2}(t-r)\|[g(r,x_{r}^{n})-g(r,x_{r}^{n-1})]Q^{1/2}\|^{2}_{\mathcal{L}(U,L^{2}(P,H))}\mathrm{d}r\\
  &&\leq L(CM)^{2}\mathbb{E}\int_{0}^{t}\|x_{r}^{n}-x_{r}^{n-1}\|_{\mathcal{C}}^{2}\mathrm{d}r,
    \end{eqnarray*}
\begin{eqnarray*}
&&\mathbb{E}\sup_{0\leq s\leq t} \int_{0}^{t}\int_{|u|<1}S_{\alpha}(s-r)[F(r,x^{n}_{r^{-}},u)-F(r,x^{n-1}_{r^{-}},u)]\tilde{N}(\mathrm{d}r,\mathrm{d}u)|^{2}\\
  &&\leq L \mathbb{E}\int_{0}^{t}S_{\alpha}^{2}(t-r)\|x^{n}_{r^{-}}-x^{n-1}{r^{-}}\|_{\mathcal{C}}^{2}\mathrm{d}r\\
  &&\leq L(CM)^{2}\mathbb{E}\int_{0}^{t}\|x^{n}_{r^{-}}-x^{n-1}{r^{-}}\|_{\mathcal{C}}^{2}\mathrm{d}r,
    \end{eqnarray*}
and
  \begin{eqnarray*}
&&\mathbb{E}\sup_{0\leq s\leq t}|\int_{0}^{s}\int_{|u|\geq1}S_{\alpha}(s-r)[G(r,x^{n}_{r^{-}},u)-G(r,x^{n-1}_{r^{-}},u)]N(\mathrm{d}r,\mathrm{d}u)|^{2}\\
&&\leq 2CM\int_{0}^{t}S_{\alpha}(t-r)\mathrm{d}r\int_{|u|\geq 1}\nu(\mathrm{d}u)\mathbb{E}\int_{0}^{t}\int_{|u|\geq 1}|G(r,x^{n}_{r^{-}},u)-G(r,x^{n-1}_{r^{-}},u)|^{2}\nu(\mathrm{d}u)\mathrm{d}r\\
&&~~+2(CM)^{2}\int_{0}^{t}\mathbb{E}\int_{|u|\geq1}|G(r,x^{n}_{r^{-}},u)-G(r,x^{n-1}_{r^{-}},u)|^{2}\nu(\mathrm{d}u)\mathrm{d}s\\
&&\leq (2(CM)^{2} \frac{|\mu|^{-1/\alpha}\pi}{\alpha\sin(\pi/\alpha)}b +2(CM)^{2})L\mathbb{E}\int_{0}^{t}\|x^{n}_{r^{-}}-x^{n-1}{r^{-}}\|_{\mathcal{C}}^{2}\mathrm{d}r,
  \end{eqnarray*}
 we get
\begin{eqnarray*}
  \mathbb{E}(\sup_{0\leq s\leq t}|x^{n+1}(s)-x^{n}(s)|^{2})\leq c \int_{0}^{t}\mathbb{E}(\sup_{0\leq r\leq s}|x^{n}(r)-x^{n-1}(r)|^{2})\mathrm{d}r
  \end{eqnarray*}
where $c=\frac{4 L(CM)^{2}}{(1-k_{0})^{2}}(\frac{|\mu|^{-1/\alpha}\pi}{\alpha\sin(\pi/\alpha)}+4+2 \frac{|\mu|^{-1/\alpha}\pi}{\alpha\sin(\pi/\alpha)}b )L.$
Note that $$\mathbb{E}\sup_{0\leq s\leq t }|x^{1}(s)-x^{0}(s)|^{2}\leq c_{0}$$
where $c_{0}=5[k_{0}^{2}c'+Lc'(1+2b)(CM\frac{|\mu|^{-1/\alpha}\pi}{\alpha\sin(\pi/\alpha)})^{2}+4Lc'(CM)^{2}\frac{|\mu|^{-\frac{1}{2\alpha}}\pi}{2\alpha\sin(\pi/2\alpha)}]$ is a positive number,
by induction we get,
\begin{equation}\label{ind1}
  \mathbb{E}(\sup_{0\leq s\leq t}|x^{n+1}(s)-x^{n}(s)|^{2})\leq c_{0}\frac{(c_{6}t)^{n}}{n!}.
\end{equation}
Taking $t=T$ in (\ref{ind1}), we have
\begin{equation}
  \mathbb{E}(\sup_{0\leq t\leq T}|x^{n+1}(t)-x^{n}(t)|^{2})\leq c_{0}\frac{(c_{6}T)^{n}}{n!}.
\end{equation}
Hence
$$P\{\sup_{0\leq t\leq T}|x^{n+1}(t)-x^{n}(t)|>\frac{1}{2^{n}}\}\leq \frac{c_{0}[c_{6}T]^{n}}{n!}.$$
Note that $\sum_{n=0}^{\infty}\frac{c_{0}[c_{6}T]^{n}}{n!}<\infty,$
by using the Borel-Cantelli lemma, we can get a stochastic process $x(t)$ on $[0,T]$ such that $x_{n}(t)$ uniformly converges to $x(t)$ as $n\rightarrow\infty$ almost surely.

It is easy to check that $x(t)$ is a unique mild solution of (\ref{1}). The proof of the theorem is complete.

\end{proof}

\section{Existence of almost automorphic mild solutions}
\begin{lemma}\label{le0}
If $x(t)\in SAP_{\omega}(L^{2}(P,H)$ and $T(t-s)\in \mathcal{L}(\mathbf{R},\mathbf{R})$ then $\Gamma_{1}(t)=\int_{0}^{t}T(t-s)x(s)\mathrm{d}s\in SAP_{\omega}(L^{2}(P,H)).$
\end{lemma}
The proof process is similar to that of Lemma 1 in \cite{dimbour2012s}, so we omit it.
\begin{lemma}\label{lemone}
If $x(t)\in SAP_{\omega}(L^{2}(P,H)),$ then $$\Gamma_{2}(t)=\int_{0}^{t}S_{\alpha}(t-s)x(s)\mathrm{d}w(s)\in SAP_{\omega}(L^{2}(P,H)).$$
\end{lemma}
\begin{proof}It is obvious that $\Gamma_{2}(t)$ is $L^{2}$-continuous.
Since
\begin{eqnarray*}
&&\|\Gamma_{2}(t+\omega)-\Gamma_{2}(t)\|^{2}\\
&& =\|\int_{0}^{t+\omega}S_{\alpha}(t+\omega-s)x(s)\mathrm{d}w(s)- \int_{0}^{t}S_{\alpha}(t-s)x(s)\mathrm{d}w(s)\|^{2}\\
&&=2\|\int_{0}^{\omega}S_{\alpha}(t+\omega-s)x(s)\mathrm{d}w(s)\|^{2}\\&&~~+2\|\int_{\omega}^{t+\omega}S_{\alpha}(t+\omega-s)x(s)\mathrm{d}w(s)-\int_{0}^{t}S_{\alpha}(t-s)x(s)\mathrm{d}w(s)\|^{2}\\
&&\leq2\|\int_{0}^{\omega}S_{\alpha}(t+\omega-s)x(s)\mathrm{d}w(s)\|^{2}\\&&~~+2\|\int_{0}^{t}S_{\alpha}(t-s)(x(s+\omega)-x(s))\mathrm{d}w(s)\|^{2}.\end{eqnarray*}
Since $x(t)\in SAP_{\omega}(L^{2}(P,H)),$ for any $\epsilon>0,$ we can choose $T_{\epsilon}>0$ such that when $t>T_{\epsilon},$ $\|x(t+\omega)-x(t)\|<\epsilon.$
For the above $\epsilon,$ we have
\begin{eqnarray*}
&&2\|\int_{0}^{t}S_{\alpha}(t-s)(x(s+\omega)-x(s))\mathrm{d}w(s)\|^{2}\\
&&\leq4\int_{0}^{T_{\epsilon}}\|S_{\alpha}(t-s)\|^{2}\|x(s+\omega)-x(s)\|^{2}\mathrm{d}s\\&&~~+4\int_{T_{\epsilon}}^{t}\|S_{\alpha}(t-s)\|^{2}\|x(s+\omega)-x(s)\|^{2}\mathrm{d}s.
\end{eqnarray*}
Note that\begin{eqnarray*}
     2\|\int_{0}^{\omega}S_{\alpha}(t+\omega-s)x(s)\mathrm{d}w(s)\|^{2}\leq 2\frac{(CM)^{2}}{1+|\mu|^{2}t^{2\alpha}}\int_{0}^{\omega} \|x(s)\|^{2}\mathrm{d}s\rightarrow0,~t\rightarrow\infty,
         \end{eqnarray*}
 that$$\int_{0}^{T_{\epsilon}}\|S_{\alpha}(t-s)\|^{2}\|x(s+\omega)-x(s)\|^{2}\mathrm{d}s\leq 4\frac{(CM)^{2}}{1+|\mu|^{2}(t-T_{\epsilon})^{2\alpha}}\|x\|_{\infty}T_{\epsilon}\rightarrow0,~t\rightarrow\infty,$$
 and that
$$\int_{T_{\epsilon}}^{t}\|S_{\alpha}(t-s)\|^{2}\|x(s+\omega)-x(s)\|^{2}\mathrm{d}s\leq \epsilon^{2} \frac{(CM)^{2}|\mu|^{-2/\alpha}\pi}{2\alpha\sin(\pi/2\alpha)},$$
we get $\lim_{t\rightarrow\infty}\|\Gamma_{2}(t+\omega)-\Gamma_{2}(t)\|=0.$ So $\Gamma_{2}(t)\in SAP_{\omega}(L^{2}(P,H)).$
\end{proof}
The following lemma is obvious by using Lemma \ref{lem1}, Lemma \ref{lem2} and the similar discussion as that for Lemma \ref{lemone}.
\begin{lemma}\label{lemtwo}
If $x(t)\in SAP_{\omega}(L^{2}(P,H))$ and $F:\mathbf{R}^{+} \times \mathcal{C}\times U\rightarrow L^{2}(P,H)$ is uniformly Poisson square-mean S-asymptotically $\omega$-periodic in $t$ on bounded sets of $\mathcal{C},$ then $$\Gamma_{3}(t)=\int_{0}^{t}\int_{|u|_{U}<1}S_{\alpha}(t-s)F(s,x_{s},u)\tilde{N}(\mathrm{d}u,\mathrm{d}s)\in SAP_{\omega}(L^{2}(P,H))$$ and $$\Gamma_{4}(t)=\int_{0}^{t}\int_{|u|_{U}\geq1}S_{\alpha}(t-s)G(s,x_{s},u)N(\mathrm{d}u,\mathrm{d}s)\in SAP_{\omega}(L^{2}(P,H)).$$
\end{lemma}
    \begin{theorem}
Assume that (H1)-(H3) are satisfied and $h, f, g$ are uniformly square-mean S-asymptotically $\omega$-periodic on bounded sets of $\mathcal{C}$. $F,G$ are uniformly Poisson square-mean S-asymptotically $\omega$-periodic on bounded sets of $\mathcal{C}$. Then (\ref{1}) has a unique S-asymptotically $\omega$-periodic solution in distribution if $$5k_{0}^{2}+5(CM)^{2}L(\frac{|\mu|^{-1/\alpha}\pi}{\alpha\sin(\pi/\alpha)})^{2}(1+b)+20L(CM)^{2}\frac{|\mu|^{-1/2\alpha}\pi}{2\alpha\sin(\pi/2\alpha)}<1.$$
\end{theorem}
\begin{proof}
Let's first show the existence of the square-mean S-asymptotically $\omega$-periodic solution of (\ref{1}), so we consider the operator $\Phi$ acting on the Banach space $SAP_{\omega}(L^{2}(P,H))$ given by
\begin{eqnarray}
\nonumber \Phi x(t)&=&S_{\alpha}(t)(\phi(0)+h(0,\phi))-h(t,x_{t})+\int_{0}^{t}S_{\alpha}(t-s)f(s,x_{s})\mathrm{d}s\\
&&+\int_{0}^{t}S_{\alpha}(t-s)g(s,x_{s})\mathrm{d}w(s)\\
\nonumber&&+\int_{0}^{t}\int_{|u|<1}S_{\alpha}(t-s)F(s,x_{s^{-}},u)\tilde{N}(\mathrm{d}s,\mathrm{d}u)\\
\nonumber&&+\int_{0}^{t}\int_{|u|\geq1}S_{\alpha}(t-s)G(s,x_{s^{-}},u)N(\mathrm{d}s,\mathrm{d}u).
\end{eqnarray}
From previous assumption one can easily see that $ \Phi x(t)$ is well defined and $L^{2}$-continuous. Moreover, from Lemma \ref{lem1}, Lemma \ref{le0}, Lemma \ref{lemone}, and Lemma \ref{lemtwo} we infer that $ \Phi $ maps $SAP_{\omega}(L^{2}(P,H))$ into itself. Next, we prove that $\Phi$ is a strict contraction on $SAP_{\omega}(L^{2}(P,H)).$ Indeed, for $x,~\tilde{x}\in SAP_{\omega}(L^{2}(P,H)),$ we get
\begin{eqnarray*}
&& \|\Phi \bar{x}(t)-\Phi \tilde{x}(t)\|^{2}\\
&&\leq
5k_{0}^{2}\|\bar{x}_{t}-\tilde{x}_{t}\|_{\mathcal{C}}^{2}+5(CM)^{2}L(\frac{|\mu|^{-1/\alpha}\pi}{\alpha\sin(\pi/\alpha)})^{2}\|\bar{x}(t)- \tilde{x}(t)\|^{2}_{\infty}\\
&&+20L(CM)^{2}\frac{|\mu|^{-1/2\alpha}\pi}{2\alpha\sin(\pi/2\alpha)}\|\bar{x}(t)- \tilde{x}(t)\|^{2}_{\infty}\\
&&+10L(CM)^{2}(\frac{|\mu|^{-1/\alpha}\pi}{\alpha\sin(\pi/\alpha)})^{2}b \|\bar{x}(t)- \tilde{x}(t)\|^{2}_{\infty}\\
&&=[5k_{0}^{2}+5(CM)^{2}L(\frac{|\mu|^{-1/\alpha}\pi}{\alpha\sin(\pi/\alpha)})^{2}(1+b)+20L(CM)^{2}\frac{|\mu|^{-1/2\alpha}\pi}{2\alpha\sin(\pi/2\alpha)}
]\\
&&~~\times\|\bar{x}(t)- \tilde{x}(t)\|_{\infty}^{2}.
\end{eqnarray*}
Since $5k_{0}^{2}+5(CM)^{2}L(\frac{|\mu|^{-1/\alpha}\pi}{\alpha\sin(\pi/\alpha)})^{2}(1+b)+20L(CM)^{2}\frac{|\mu|^{-1/2\alpha}\pi}{2\alpha\sin(\pi/2\alpha)}<1,$ it follows that $\Phi$ is a contraction mapping on $SAP_{\omega}(L^{2}(P,H)).$ By the classical Banach fixed-point principle, there exists a unique $x\in SAP_{\omega}(L^{2}(P,H))$ such that $\Phi x=x,$ which is the unique square-mean S-asymptotically $\omega$-periodic solution of (\ref{1}). By Lemma \ref{lem3}, we deduce that (\ref{1}) has a unique S-asymptotically $\omega$-periodic solution in distribution.
The proof is now complete.
\end{proof}

\section*{Acknowledgments}
This work was supported by the National Science Foundation of China (No.11071050).


\section*{References}
\begin{thebibliography}{99}
\bibitem{yoshizawa2012stability}
T. Yoshizawa, Stability theory and the existence of periodic solutions and almost periodic solutions (Vol. 14), Springer Science \& Business Media, 2012.
\bibitem{n2013almost}
 G. M. N'Gu\'{e}r\'{e}kata, Almost automorphic and almost periodic functions in abstract spaces, Springer Science \& Business Media, 2013.

\bibitem{corduneanu1989almost}

 C. Corduneanu, Almost periodic functions, Chelsea Pub Co, 1989.



\bibitem{hino2001almost}
 Y. Hino, T. Naito, N. VanMinh, J. S. Shin,  Almost periodic solutions of differential equations in Banach spaces, CRC Press, 2001.
\bibitem{henriquez2015almost}
 H. R. Henr\'{\i}quez, C. Cuevas, A. Caicedo, Almost periodic solutions of partial differential equations with delay, Adv. Difference Equ., 2015(1)(2015) 1-15.

\bibitem{diagana2007existence}
T. Diagana, E. M. Hern\'{a}ndez, Existence and uniqueness of pseudo almost periodic solutions to some abstract partial neutral functional¨Cdifferential equations and applications, J. Math. Anal. Appl., 327(2)(2007) 776-791.

\bibitem{diagana2009existence}
T. Diagana, Existence of weighted pseudo almost periodic solutions to some classes of hyperbolic evolution equations, J. Math. Anal. Appl., 350(1)(2009) 18-28.

\bibitem{diagana2016existence}
 T. Diagana, H. Zhou, Existence of positive almost periodic solutions to the hematopoiesis model, Appl. Math. Comput., 274(2016) 644-648.
\bibitem{henriquez2016pseudo}
 H. R Henr\'{\i}quez, M. Pierri, V. Rolnik, Pseudo S-asymptotically periodic solutions of second-order abstract Cauchy problems, Appl. Math. Comput., 274(2016) 590-603.
\bibitem{fu2010square}
 M. Fu, Z. Liu, Square-mean almost automorphic solutions for some stochastic differential equations, Proc. Amer. Math. Soc., 138(10)(2010) 3689-3701.



\bibitem{cao}
 J. Cao, Q. Yang, Z. Huang, Existence and exponential stability of almost automorphic mild solutions for stochastic functional differential equations, Stochastics, 83(03)(2011) 259-275.
\bibitem{fu2012almost}
M. Fu, Almost automorphic solutions for nonautonomous stochastic differential equations, J. Math. Anal. Appl., 393(1)(2012) 231-238.
\bibitem{sun}
 Z. Liu, K. Sun, Almost automorphic solutions for stochastic differential equations driven by L\'{e}vy noise, J. Funct. Anal., 266(3)(2014) 1115-1149.
\bibitem{li2015weighted}
  K. Li, Weighted pseudo almost automorphic solutions for nonautonomous SPDEs driven by L\'{e}vy noise, J. Math. Anal. Appl., 427(2)(2015) 686-721.


\bibitem{henriquez2008s}
H. R. Henr\'{\i}quez, M. Pierri, P. T\'{a}boas, On S-asymptotically $\omega$-periodic functions on Banach spaces and applications, J. Math. Anal. Appl., 343(2)(2008) 1119-1130.
\bibitem{zhongchao1987asymptotically}
Z. Liang, Asymptotically periodic solutions of a class of second order nonlinear differential equations, Proc. Amer. Math. Soc.,(1987) 693-699.
\bibitem{grimmer1969asymptotically}
R. C. Grimmer, Asymptotically almost periodic solutions of differential equations, SIAM J. Appl. Math., 17(1)(1969) 109-115.
\bibitem{haiyin2006massera}
H. Gao, K. Wang, F. Wei, X. Ding, Massera-type theorem and asymptotically periodic logistic equations, Nonlinear Anal. real world applications, 7(5)(2006) 1268-1283.

\bibitem{cuevas2009s}
 C. Cuevas, J. C. de Souza, S-asymptotically $\omega$-periodic solutions of semilinear fractional integro-differential equations, Appl. Math. Lett., 22(6)(2009) 865-870.
\bibitem{cuevas2010existence}
 C. Cuevas, J. C. de Souza, Existence of S-asymptotically $\omega$-periodic solutions for fractional order functional integro-differential equations with infinite delay, Nonlinear Anal.  Theory, Methods \& Applications, 72(3)(2010) 1683-1689.

\bibitem{dimbour2014s}
W. Dimbour, J. C. Mado, S-asymptotically $\omega$-periodic solution for a nonlinear differential equation with piecewise constant argument in a Banach space., Cubo (Temuco), 16(3)(2014) 55-65.



\bibitem{cuesta2003numerical}
E. Cuesta, C. Palencia, A numerical method for an integro-differential equation with memory in Banach spaces: Qualitative properties, SIAM J. Numer. Anal., 41(4)(2003) 1232-1241.
 \bibitem{dos2010asymptotically}
  J. P. C. Dos Santos, C. Cuevas, Asymptotically almost automorphic solutions of abstract fractional integro-differential neutral equations, Appl. Math. Lett., 23(9)(2010) 960-965.



\bibitem{applebaum2009levy}

 D. Applebaum, L\'{e}vy processes and stochastic calculus, Cambridge university press, 2009.
\bibitem{holden1996stochastic}
H. Holden, B. {\O}ksendal, J. Ub{\o}e, T. Zhang, Stochastic partial differential equations (pp. 141-191), Birkh\"{a}user Boston, 1996.
\bibitem{albeverio2002stochastic}
S. Albeverio, B. R\"{u}diger,  Stochastic integrals and L\'{e}vy-Ito decomposition on separable Banach spaces, In 2nd MaPhySto L\'{e}vy Conference (p. 9), 2002.
\bibitem{da2014stochastic} G. Da Prato, J. Zabczyk, Stochastic equations in infinite dimensions (Vol. 152), Cambridge university press, 2014.



\bibitem{haase2006functional}
 M. Haase, The functional calculus for sectorial operators (pp. 19-60), Birkh\"{a}user Basel, 2006.

\bibitem{lunardi2012analytic}
 A. Lunardi, Analytic semigroups and optimal regularity in parabolic problems, Springer Science \& Business Media, 2012.


\bibitem{bajlekova2001fractional}
 E. G. Bajlekova, Fractional evolution equations in Banach spaces (Doctoral dissertation, University Press Facilities, Eindhoven University of Technology), 2001.
\bibitem{cuesta2007asymptotic}
 E. Cuesta, Asymptotic behaviour of the solutions of fractional integro-differential equations and some time discretizations, Discrete Contin. Dyn. Syst.,(2007) 277-285.
\bibitem{dimbour2012s}
 W.Dimbour, G. M. N'Gu\'{e}r\'{e}kata, S-asymptotically $\omega$-periodic solutions to some classes of partial evolution equations, Appl. Math. Comput., 218(14)(2012) 7622-7628.






\end{thebibliography}
\end{document}